\documentclass[11pt]{amsart}

\usepackage[english]{babel}

\usepackage{color}
\usepackage{array}
\usepackage{enumerate}
\usepackage{enumitem}
\usepackage{refcount}
\usepackage{algorithm}
\usepackage{algorithmic}
\usepackage{boxedminipage}
\usepackage{float}
\usepackage[utf8x]{inputenc}
\usepackage{soul}

\usepackage{parskip}

\usepackage[colorlinks]{hyperref}
\hypersetup{
	colorlinks,
	linkcolor={red!50!black},
	citecolor={green!50!black},
	urlcolor={blue!80!black}
}
\usepackage[nameinlink,capitalize]{cleveref}

\usepackage{color}
\usepackage{graphicx}
\usepackage{multirow}
\usepackage{tikz}
\usepackage{amsmath}
\usepackage{amssymb}
\usepackage{caption}
\usepackage{tikz-cd}

\newtheorem{thm}{Theorem}[section]
\newtheorem{prop}[thm]{Proposition}
\newtheorem{lem}[thm]{Lemma}
\newtheorem{cor}[thm]{Corollary}

\theoremstyle{definition}
\newtheorem{defn}[thm]{Definition}

\newtheorem{exa}[thm]{Example}

\newtheorem{remark}[thm]{Remark}

\def\proof{{\bf Proof.}}

\newcommand{\bA}{\mathcal A}
\newcommand{\bB}{\mathcal B}
\newcommand{\bC}{\mathcal C} 

\newcommand{\bO}{\mathcal O}
\newcommand{\bL}{\mathcal L}

\newcommand{\bZ}{\mathcal Z}

\newcommand{\bI}{\mathcal I}

\DeclareMathOperator{\Hilb}{Hilb}
\DeclareMathOperator{\im}{Im}
\DeclareMathOperator{\reg}{reg}

\newcommand{\prfend}{\hbox to7pt{\hfil}
	
	\par\vskip-\baselineskip\hbox to\hsize
	{\hfil\vbox {\hrule width6pt height6pt}}\vskip\baselineskip}

\def\dd{\medskip \par \noindent}
\long\def\eatit#1{}

\def\Z{\mathbb{Z}}
\def\N{\mathbb{N}}
\def\C{\mathbb{C}}
\def\A{\mathbb{A}}
\def\P{\mathbb{P}}

\def \"{``}

\font\tengothic=eufm10
\font\sevengothic=eufm7
\newfam\gothicfam
\textfont\gothicfam=\tengothic
\scriptfont\gothicfam=\sevengothic
\def\goth#1{{\fam\gothicfam #1}}

\newcommand{\gm}{\goth m}

\usepackage{xypic}

\DeclareMathOperator{\Res}{Res}

\begin{document}
		
	\author[E. Ballico]{Edoardo Ballico}
	\address[E. Ballico]{Dipartimento di Matematica, Universit\`a di Trento, Via Sommarive, 14 - 38123 Povo (Trento), Italy}
	\email{edoardo.ballico@unitn.it}
	
	\title{Postulation of schemes of length at most 4 on surfaces}
	\author[S. Canino]{Stefano Canino}
	\address[S. Canino]{Wydział Matematyki, Informatyki i Mechaniki, Uniwersytet Warszawski, ul.~Stefana Banacha	2, 02-097 Warsaw, Poland.}
	\email{s.canino@uw.edu.pl}
	
	\maketitle


	\begin{abstract}
	In this paper we address the postulation problem of zero-dimensional schemes on a surface of length at most 4. We prove some general results and then we focus on the case of $\P^2$, $\P^1\times\P^1$ and  Hirzebruch surfarces. In particular, we prove that except for few well-known exceptions, a general union of schemes of length at most 4 has always good postulation in $\P^2$ and in $\P^1\times\P^1$.
	\end{abstract}
	
	
	\section{Introduction}
	In \cite{ccgi} the authors considered an interesting class of zero-dimensional schemes of $\P^n$ which they called {\it m-symmetric schemes}. More specifically, $X\subset \P^n$ is an $m$-symmetric scheme if it is supported at one point $P\in\P^n$ and $\ell(L\cap X)=m$ for any line $L$ through $P$, where $\ell(L\cap X)$ denotes the length of $L\cap X$. Among all the $m$-symmetric schemes, they studied in detail the specific class of 2-squares, i.e. zero-dimensional schemes of $\P^2$ which are isomorphic to the scheme defined by the intersection of two double lines and they showed that, except for double points, 2-squares are the only 2-symmetric schemes of $\P^2$. Moreover, despite being interesting per se, these schemes  have some interesting connections with the Jacobian schemes of plane algebraic curves, which indeed were the starting point for their definition, see \cite{cgi}, and with symmetric and partially symmetric tensors, see \cite{ccgi}. Successively, the postulation of 2-squares in $\P^2$ has been studied in \cite{ccgio}, where it is proved that a general union of 2-squares in $\P^2$ has always good postulation.
	
	Here, as in \cite{ccgio}, we consider zero-dimensional schemes embedded with an analytic isomorphism, i.e. we consider finite colength ideals $I\subset \C[[x_1,\dots ,x_n]]$ or $I\subset \C\{x_1,\dots x_n\}$; see \cite{ber}. The case of zero-dimensional schemes contained in a smooth surface, case $n=2$, is quite special and it has a very different flavour. If we fix the colength, $z$, i.e. the length of the associated zero-dimensional scheme, the set $\text{Hilb}^t\C\{x_1,x_2\}$ is an algebraic set whose reduction is irreducible and of dimension $t-1$, see \cite{bri,g1,g, iar}. Its singular locus is described in \cite{g1}. For $t\le 6$, there is a complete classification of its isomorphism classes with the description of all connected zero-dimensional schemes of length at most $6$, see \cite{bri} Chapter IV. In \cite{bri,iar} there is a huge number of other results, e.g. the irreducibility of the strata with fixed Hilbert-Samuel function and the dimension of these sets (\cite[Ch. III]{bri}). 
	
	Our pourpose in this paper is to study the postulation of unions of schemes of length less or equal then 4 on integral surfaces. As stated in \cite{bri}, an open dense set of $\text{Hilb}^t\C\{x_1,x_2\}$ is formed by the so-called curvilinear schemes, i.e. the ones supported at one point and contained in a smooth curve. However, for our pourposes, handling these schemes is trivial, due to Remark \ref{cm}, which uses \cite{cm}, thus we want to consider also schemes which are not in this open dense set. 
	
	Recall that an $m$-fat point is a zero-dimensional scheme defined by the $m$\textsuperscript{th} power of an ideal defining a point and an $m$-jet is a zero-dimensional scheme of length $m$ supported at one point and contained in a line; see Section \ref{sec:preliminaries} for more details. As proved in \cite{bri} Chapter IV, the only element of $\Hilb^1\{x,y\}$ is a  simple point,  the elements of $\Hilb^2\{x,y\}$ are only 2-jets and the elements of $\Hilb^3\{x,y\}$ are only length 3 curvilinear schemes and the double point. In $\Hilb^4\{x,y\}$ the situation is more interesting: indeed, there are length 4 curvilinear schemes and two other types of schemes, which we call {\it 2-squares} and {\it tiles}. A 2-square of an integral surface $M$, is a zero-dimensional subscheme $X\subset M$ which is isomorphic to a 2-square of $\A^2$ or, equivalentely, it is supported at $P\in M$ and for any curve $\bC\subset M$ passing through $P$ it holds that $\ell(\bC\cap X)=2$; note that, by \cite{ccgi} Proposition 2.17, this definition of 2-square and the definition given in \cite{ccgi} coincide in $\P^2$. A tile of $M$ is instead a zero-dimensional subscheme $Y\subset M$ which is isomorphic to the scheme of $\A^2$ defined by the ideal $(x_2^2,x_1x_2,x_1^3)$ or, equivalentely, it is supported at $P\in M$, it is not curvilinear and there is a curve $\bC\subset M$ passing through $P$ such that $\ell(\bC\cap Y)=3$; see Section \ref{sec:preliminaries} for more details.
	
	The set of length $4$ connected schemes containing the double point $2P$ and contained in the triple point $3P$ is exactly the set of tiles and 2-squares supported at $P$ and it is an irreducible variety of dimension $2$. Its irreducibility is a particular case of the irreducibility of the stratum of connected zero-dimensional schemes with fixed Hilbert-Samuel function. For this reason, it makes sense to deal with a general union of length 4 curvilinear schemes, tiles and 2-squares. Moreover, by the same argument, we can also deal with general unions of the previous mentioned schemes, double points and curvilinear schemes.
	
	As we said, we want to study the postulation of general unions of low length zero-dimensional schemes on surfaces. First of all, in Section \ref{sec:preliminaries} we recall some results and we prove some lemmata and remarks that allow to considerably simplify the proofs of the paper. The most important of these is Lemma \ref{2p3p}, which state that if $P$ is a point of an integral projective surface $M$, then any generic 2-square supported at $P$ imposes one condition more then $2P$ at any linear system of divisors of $M$. After that, we prove some general results for the postulation of our schemes on surfaces. In particular, given an integral non-degenerate surface $M\subset \P^n$, we show that if $\sigma_r(M)$, the $r$\textsuperscript{th} secant variety of $M$, is non defective, then the dimension of the liner span of a general union of 2-squares of $M$ has the expected dimension (see Corollary \ref{2squares}). As a consequence of that, by using some results of \cite{laf}, we are able to prove that, in many cases, a general union of 2-squares on an Hirzebruch surface has good postulation (see Proposition \ref{i4}).
	
	In Section \ref{sec:p2} we prove, in Corollary \ref{p2<4}, that a general union of any type of zero-dimensional schemes of length less or equal then 4, i.e. simple points, double points, 2-squares, tiles and curvilinear schemes of length at most 4, has good postulation, except for the well-known cases of two or five double points. Our strategy is the following: first we prove that the statement is true for a general union of tiles (Theorem \ref{tiles}), then we prove that it is true for a general union of double points and tiles (Theorem \ref{fattiles}), and finally we get our result thanks to Lemma \ref{2p3p} and the result of \cite{cm} stated in Remark \ref{cm}. In Section \ref{sec:p1p1} we prove an analogous result, Corollary \ref{p1p1<4}, in $\P^1\times\P^1$ by using the same strategy. Both in the case of $\P^2$ and in that of $\P^1\times\P^1$, we also characterise some sets of tiles having bad postulation.
	
	In our proofs we use the so-called Horace method (\cite{hir}), an inductive method which uses the residual exact sequence defined in Remark \ref{horace} and  several fine tuning to do its inductive steps. We manage to avoid using degenerations of several small zero-dimensional schemes to a single bigger scheme, which are an often used powerful tool in every dimension (\cite{ga,go}), but are unnecesary in our case.
	
	We always study the postulation of the general union of a prescribed number of certain schemes, i.e. by the semicontinuity theorem the \lq\lq best\rq\rq\, postulation. In this paper, we also briefly consider the case of the \lq\lq worst\rq\rq\, postulation for tiles in $\P^2$ or in $\P^1\times \P^1$ and we construct examples with a prescribed index of regularity in a certain interval (see Examples \ref{new1} and \ref{new3} and Propositions \ref{new2} and \ref{new2.0}). We do not consider schemes which are adapted to the ambient surface $X$. For instance, for $X=\P^1\times \P^1$ we do not consider general unions of the very particular 2-squares which are the product of a length $2$ scheme of the first factor of $\P^1\times \P^1$ and a length $2$ scheme of the second factors. When $X=\P^2$ we do not consider general unions of schemes whose connected components are of prescribed length and contained in a line, the line being different, in general, for different connected components (their postulation is described in \cite{e}).

	\section{General results on surfaces}\label{sec:preliminaries}
	We work on the field of complex numbers $\C$. Given a variety $M$ and a point $P\in M$ defined by the maximal ideal $\gm$, the $m$-fat point of $M$ supported at $P$ is the closed subscheme $mP\subset M$ defined by the ideal $\gm^m$; a 2-fat point is also called a {\it double point}. We use analogue notation for lines, i.e. if $L\subset M$ is a line, then we denote by $2L$ the corresponding double line.  Given a closed subscheme $X\subset M$ we denote  by $\bI_X$ its ideal sheaf and by $\deg(X)$ its degree. If $X$ is a zero-dimensional scheme, its degree is also called the {\it length of} $X$ and it denoted by $\ell(X)$. We use in $\A^2$ the coordinates $(x_1,x_2)$.\dd
	 We recall now the definition of good postulation of a zero-dimensional scheme.
	\begin{defn}
	Let $M$ an algebraic variety, $\bL$ a line bundle on $M$ and $X\subset M$ a zero-dimensional scheme. We say that $X$ has good postulation with respect to $\bL$ if
	$$h^0(M,\bI_X\otimes\bL)=\min\{0,h^0(M,\bL)-\ell(X)\}.$$
	\end{defn}
	Note that, having good postulation is equivalent to say that 
	$$h^0(M,\bI_X\otimes\bL)\cdot h^1(M,\bI_X\otimes\bL)=0.$$
	We will often use the following remarks and lemma.
\begin{remark}\label{horace}
Let $M$ be an integral projective variety, $X\subset M$ a zero-dimensional scheme and $D\subset M$ an effective Cartier divisor of $M$. The residual scheme $\Res_D(X)$ of $X$ with respect to $D$ is the closed subscheme of $M$ with $\bI_X:\bI_D$ as its ideal sheaf. We have $\Res_D(X)\subseteq X$ and $\ell(X) =\ell(X\cap D) +\ell(\Res_D(X))$. For all the line bundles $\bL$ on $M$ we have an exact sequence
$$
	0\to \bI_{\Res_D(X)}\otimes \bL(-D)\to \bI_X\otimes \bL\to \bI_{X\cap D,D}\otimes \bL_{|D}\to 0
$$
which is called {\it the exact sequence of $X$ with respect to $D$}.
\end{remark}
\begin{remark}\label{rangè}\rm
Let $M$ be an integral projective variety of dimension $n\ge 1$.  Let $\bL$ be a line bundle on $M$ and $V\subseteq H^0(M,\bL)$ a linear subspace. Fix zero-dimensional schemes $Z\subset W$. If $V(-Z)=0$, then $V(-W) =0$. If $\dim V(-W) =\dim V-\deg(W)$, then $\dim V(-Z) =\dim V-\deg(Z)$.
\end{remark}
\begin{remark}\label{cm}
	Let $M$ be an irreducible integral projective variety of dimension $n\ge 1$. For any \linebreak $P\in M_{\reg}$ and any positive integer $t$ the let $\bC(M,P,t)$ the set of all degree $t$ curvilinear subschemes of $M$ supported at $P$. This set $\bC(M,p,t)$ is non-empty and irreducible. For all positive integers $s$, and $t_i$, $1\le i\le s$, let $\bC(M,t_1,\dots ,t_s)$ be the set of all zero-dimensional schemes $A\subset M_{\reg}$ with $s$ connected components, say $A =A_1\cup \cdots \cup A_s$, whose connected components are curvilinear and $\deg(A_i)=t_i$ for all $i$. Since $M_{\reg}$ is irreducible and each $\bC(M,P,t)$ is irreducible, then $\bC(M,t_1,\dots ,t_s)$ is non-empty and irreducible. Hence, it makes sense to consider general unions of $s$ curvilinear schemes on $M$. The short note \cite{cm} may be stated in the following way: let $\bL$ be a line bundle on $M$ and $V\subseteq H^0(M,\bL)$ a linear subspace. Then $\dim V(-Z) =\max \{0,\dim V-t_1-\cdots-t_s\}$ for a general $Z\in \bC(T,t_1,\dots,t_s)$.
\end{remark}
\begin{lem}\label{comlem}
	Let $M$ be an integral projective variety with $\dim M>1$, $\bL$ a line bundle on $M$, $D\subset M$ an integral Cartier divisor, $Z\subset M$ a zero-dimensional scheme and $S\subset D$ a general union of $s$ points on $D$.
	\begin{itemize}[leftmargin=*]
		\item If $h^0(M,\bI_{\Res_D(Z)}\otimes \bL(-D)) \le h^0(M,\bI_Z\otimes \bL)-s$ then 
		$$h^0(M,\bI_{Z\cup S}\otimes \bL) =h^0(M,\bI_Z\otimes \bL)-s.$$
		\item If $h^0(M,\bI_{\Res_D(Z)}\otimes \bL(-D))=0$ then
		$$h^0(M,\bI_{Z\cup S}\otimes \bL)=\max\{0,h^0(M,\bI_Z\otimes \bL)-s\}.$$
	\end{itemize}
\end{lem}
\proof\, Since $\dim M>1$, $\dim D=\dim M-1\ge 1$. Obviously we have
$$h^0(M,\bI_{Z\cup S}\otimes \bL)\ge \max\{0,h^0(M,\bI_Z\otimes \bL)-s\}$$
and, by general assumption, we have $S\cap Z=\emptyset$.  If $s=1$, then the statement is true because $h^0(M,\bI_{Z\cup S}\otimes \bL) =h^0(M,\bI_Z\otimes \bL)$ if and only if $D$ is contained in the base locus
of $\bI_Z\otimes \bL$. Now, assume $s>1$ and suppose by induction that the result is true any $1\leq s'<s$. Take a general union of $s-1$ points in $D$, say $S'\subset D$. Since $\Res_D(Z\cup S') =\Res_D(Z)$, the inductive assumption
gives $h^0(M,\bI_{Z\cup S'}\otimes \bL)=\max\{0,h^0(M,\bI_Z\otimes \bL)-s+1\}$. Take as $S$ the union of $S'$ and a general point of $D$. Apply the case $s=1$ to the zero-dimensional scheme $Z\cup S'$.\prfend
\begin{remark}\label{comlem0}\rm
	Lemma \ref{comlem} is characteristic free but if we assume to work in characteristic 0 and we take a partition of $s$, say $s_1\geq s_2\geq \dots\geq s_r>0$ and curvilinear schemes of length $s_1,\dots,s_r$ on $D$, then we can apply Remark \ref{cm}. Hence, in characteristic 0 Lemma \ref{comlem} is still true by replacing $s$ general points with $r$ curvilinear schemes on $D$ of total length $s$.
\end{remark}
We recall now from \cite{bri} the classification of length 4 schemes supported at one point of a surface.
\begin{defn}\label{hilb4}\rm
	Let $M$ be an integral projective surface and $P\in M_{\reg}$. A zero-dimensional scheme $X$ supported at $P$ is said to be:
	\begin{itemize}[leftmargin=*]
		\item a {\it curvilinear scheme of length $m$}, if $X$ is isomorphic to the subscheme of $\A^2$ defined by $(x_2,x_1^n)$. Moreover, if $X$ is supported on a line, it is called an {\it $m$-jet};
		\item a {\it 2-square}, if $X$ is isomorphic to the subscheme of $\A^2$ defined by $(x_1^2,x_2^2)$;
		\item a {\it tile}, if $X$ is isomorphic to the subscheme of $\A^2$ defined by $(x_2^2,x_1x_2,x_1^3)$.
	\end{itemize}
	\end{defn}
	\begin{remark}\label{brirem}\rm
	Let $M$ be an integral projective surface and let $\Hilb^4$ be the punctual Hilbert scheme parameterising the zero-dimensional schemes $X$ supported at one point $P\in M_{\reg}$ with $\ell(X)=4$. By \cite{bri} Proposition IV.2.1, every element of $\Hilb^4$ is isomorphic to one (and only one) of the schemes of Definition \ref{hilb4}. Moreover, the dimension of $\Hilb^4$ is 3 and we have that:
	 \begin{itemize}[leftmargin=*]
	 	\item the set of curvilinear schemes is $Z_1\subset\Hilb^4$ and $Z_1$ is open and of dimension 3. Hence, the general element of $\Hilb^4$ is curvilinear;
	 	\item the set of 2-squares and tiles is $Z_2\subset\Hilb^4$ and $Z_2$ has dimension 2;
	 	\item the set of tiles is $K\subset Z_2$ and $K$ is isomorphic to $\P^1$. Hence, the general element of $Z_2$ is a 2-square. 
	 \end{itemize}
	Note that, by \cite{bri} p. 76 and by \cite{ccgi} Proposition 3.1, we can characterise curvilinear schemes, 2-squares and tiles by looking at their intersections with smooth curves. In particular we have that a scheme $X$ supported at $P\in M_{\reg}$ with $\ell(X)=4$ is:
	\begin{itemize}[leftmargin=*]
		\item a curvilinear scheme, if there exists a smooth curve $\bC$ such that $\ell(\bC\cap X)=4$, i.e. $X\subset \bC$;
		\item a 2-square, if $\ell(\bC\cap X)=2$ for any smooth curve passing through $P$;
		\item a tile, if there exist a smooth curve $\bC$ such that $\ell(\bC\cap X)=3$, and we call such a curve a {\it long side of $X$}.
	\end{itemize}
	Note that if $X\subset \P^2$ or $X\subset \P^1\times\P^1$, there exists a line that is a long side of $X$.
	\end{remark}
	Given a point $P\in M$ we denote by $\bZ(M,P)$ the set of all the zero-dimensional schemes $X$ such that $2P\subset X\subset 3P$ and $\ell(X)=4$. Note that, up to isomorphism, $\bZ(M,P)=Z_2$, thus its generic element is a 2-square.\dd
	The following lemma will be essential to simplify our proofs on 2-squares.
	\begin{lem}\label{2p3p}
	Let $M$ be an integral projective surface, $P\in M_{\reg}$, $\bL$ a line bundle on $M$ and \linebreak $V\subseteq H^0(X,\bL)$ a vector subspace. If $X\in\bZ(M,P)$ is a general 2-square then $$\dim V(-X)=\max\{0,\dim V(-2P)-1\}.$$
	\end{lem}
	\proof\,
	If $\dim V(-2P)=0$ then the statement is trivial, so we assume $\dim V(-2P)>0$. By \cite{cc1} Proposition 2.3, we have $\dim V(-3P)<\dim V(-2P)$. Suppose by contradiction that for any 2-square $X\in\bZ(M,P)$, $\dim V(-2P)=\dim V(-X)$. By \cite{ccgi} Theorem 2.17, the union of all the 2-squares supported at $P$ is $3P$ and thus $\dim V(-2P)=\dim V(-3P)$, a contradiction. Hence, there exists a 2-square $X\in\bZ(M,P)$ such that $\dim V(-X)=\dim V(2P)-1$ and the same holds for a generic 2-square in $\bZ(M,P)$.\prfend
	
	The proof of Lemma \ref{2p3p} and Remark \ref{brirem} give the following result.
	\begin{remark}\label{addo}
		Let $X$ be a projective surface, $P\in X_{\reg}$, $\bL$ a line bundle on $X$ and $V\subseteq H^0(X,\bL)$ a linear subspace. If $V(-3P) \nsubseteq V(-2P)$,
		then $V(-Y)\subsetneq V(-2P)$ for a general 2-square $Y$ supported at $P$. Note that $\deg(3P) =6 =3+\deg(2P)$. The closure of the union of all tiles $Z$ supported at $P$
		is a closed subscheme of $3P$ with degree at least $5$. Hence, if $\dim V(-3P) \le \dim V(-2P) -2$, then $V(-Z) \subsetneq V(-2P)$ for a general tile $Z$ supported at $P$.
	\end{remark}
	Some of the results on tiles obtained in the next sections may also be proved using the second part of Remark \ref{addo}.
	As an easy consequence of Lemma \ref{2p3p}, we can prove the following theorem. Given a projective variety $M$, we denote by $\sigma_r(M)$ its $r$\textsuperscript{th} secant variety.
	\begin{thm}\label{span}
	Fix a positive integer $r$ and consider $M\subset\P^n$, an integral non-degenerate surface such that $\dim \sigma_r(M)=3r-1$. Fix an integer $s$ such that $0\leq s\leq r$ and let $X\subset M_{\reg}$ be a general union of $r-s$ double points and $s$ 2-squares. Then
	$$\dim\langle X\rangle=\min\{n,3r+s-1\}$$
	where $\langle X\rangle$ is the smallest projective subspace of $\P^n$ containing $X$.
	\end{thm}
	\proof\, We prove the theorem by induction on $s$. Let $Z\subset M_{\reg}$ be a general union of $r$ double points of $M$, say $Z=2P_1\cup\dots\cup2P_r$. By Terracini's Lemma, the assumption $\dim\sigma_r(M)=3r-1$ is equivalent to $\dim\langle Z\rangle=3r-1$, and this prove the base case $s=0$. Now suppose by induction that the statement is true for any $0\leq s'<s$. For any $i=1,\dots,s-1$ let $Y_i\in\bZ(M,P_i)$ be a general two squares and consider 
	$$X'=Y_1\cup\dots\cup Y_{s-1}\cup2P_s\cup\dots\cup2P_r.$$
	Hence, recalling that $\dim\langle X^\prime\rangle=n-h^0(\P^n,\bI_{X^\prime,\P^n}(1))$, the result follows by applying Lemma \ref{2p3p} to $P_s$. \prfend
	As an immediate consequence of Theorem \ref{span}, we have the following two corollaries.
	\begin{cor}\label{2squares}
	Fix a positive integer $r$ and consider $M\subset\P^n$, an integral non-degenerate surface such that $\dim \sigma_r(M)=3r-1$, and $X\subset M_{\reg}$ a general union of $s$ 2-squares with $0\leq s\leq r$. Then $\dim\langle X\rangle=\min\{n,4s-1\}$. Moreover, if $3r-1\geq n$, then $\dim\langle X\rangle=\min\{n,4s-1\}$ for any $s\in\N$.
	\end{cor}
	\begin{cor}\label{2squares2}
	Let $M\subset\P^n$ be an integral non-degenerate surface and fix $r,s\in\N$. Let $X\subset M_{\reg}$ be a general union of $r$ double points and $s$ 2-squares and $Y\subset M_{\reg}$ be a general union of $r+s$ double points. Then, $\dim\langle X\rangle=\min\{n,3r+4s-1\}$ if one of the following conditions is satisfied:
	\begin{enumerate}[leftmargin=*]
	\item $\dim\langle Y\rangle\geq n-s$;
	\item $\dim\langle Y\rangle=3(r+s)-1$.
	\end{enumerate}
	\end{cor}
	As an example of the use of Corollary \ref{2squares}, we describe the Hilbert function of a general union of 2-squares on any complete very ample linear system on any Hirzebruch surface $F_e$, $e\geq 0$, using the lists of their defective embeddings, see \cite{laf} Proposition 4.1 and Proposition 5.1. We use the following notation: let $F_e$, $e\ge 0$, be the Hirzebruch surface with a section $h$ of one of its ruling with self-intersection $-e$ (the ruling is unique if $e>0$). We have
	$\mathrm{Pic}(F_e)\cong \Z^2$ and we take as a basis of it $h$ and a fiber $f$ of a ruling of $F_e$ with $h\cdot f =1$. The line bundle $\bO_{F_e}(ah+bf)$ is very ample if and only if $a>0$ and $b>ae$. In the next proposition we restrict  to very ample line bundles, but for $e>0$ A. Laface in \cite{laf} also considers the complete linear systems with $b=ae$, $a>0$, which for $e\ge 2$ have as images a cone over a rational normal curve of $\P^e$, while if $e=1$ they are the linear systems $|\bI_{aP,\P^2}(a)|$ with $P\in\P^2$, and these results may be applied even to these linear systems and to non-complete linear systems using Lemma \ref{2p3p} (for instance, see Remark \ref{i5}).
	\begin{prop}\label{i4}
	Let $a,b,e\in\N$ with $a\geq 2, e>0$ and $b>ae$. Set
	$$r=ab -\frac{a(a-1)e}{2}-1$$
	and let $X\subset \P^r$ be the embedding of $F_e$ by the complete linear system $|\bO_{F_e}(ah+bf)|$. Let $W(s)\subset X$, $s>0$, be a general union of $s$ 2-squares of $X$. Then either $h^0(\bI_{W(s)}(ah+bf)) =0$ or \linebreak $h^1(\bI_{W(s)}(ah+bf)) =0$.
	\end{prop}
	\proof\,Let $Z(s)\subset F_e$ be a general union of $s$ 2-points. By Corollary \ref{2squares} and \cite{laf} Proposition 4.1 and Proposition 5.1, for any $e\ge 0$ we need to check the statement for the linear systems $|\bO_{F_e}(2h+bf)|$, with $b$ even,  and $s =b+e+1$. In this case $h^0(\bI_{W(b+e+1)}(2h+bf)) =0$, because $h^0(\bI_{Z(b+e+1)}(2h+bf))=1$. \prfend
	\begin{remark}\label{au0001}
		If $e\ge 3$,  $a=1$ and $b=e$, then $\dim =|\bO_{F_e}(h+ef)| = e+1$ and the image $T\subset \P^{e+1}$ of $F_e$ by the complete linear system $|\bO_{F_e}(h+ef)|$ is a cone over a rational normal curve of $\P^{e+1}$. Hence $\dim \sigma_s(T) = \min\{e+1,2s\}$. Hence in this case even certain general unions of 2squares have not maximal rank.
	\end{remark}
	\begin{remark}\label{i5}
		Fix integers $a>m\ge 3$ and $s>0$. Fix $P\in \P^2$. Let $W(s)$ be a a general union of $s$ 2-squares. Let $u: F_1\to \P^2$ denote the blow-up of $\P^2$ at $P$. The map $u$ induces an isomorphism between the  linear system $|\bI_{mP}(a)|$ and the linear system $|(a-m)h+af|$ on $F_1$. Hence Proposition \ref{i4} gives
		that either $h^0(\bI_{mP\cup W(s)}(a)) =0$ or $h^1(\bI_{mP\cup W(s)}(a)) =0$.
		If $m>a$, then obviously $h^0(\P^2,\bI_{mP}(a)) =0$. Now assume $a=m$.
		By the list in \cite{laf}, Table 1, we need to check the case $a=m=4$ with $s=5$. In this case Lemma \ref{2p3p} gives $h^0(\bI_{4P\cup W(5)}(4)) =0$.
		In the same way we get the quadratic cone $X\subset \P^3$, image of the map $F_2\to X$. The quadratic cone is defective only with respect to $\bO_X(2)$, when consider 3 double points and its defect is $1$.
		\end{remark}
		The following example shows that, on other surfaces, not all general unions of 2-squares have good postulation. It is constructed using \cite{cat,DeP} or \cite{cc1} Example 4.3.
		\begin{exa}\label{z2}
		Fix integers $t\ge 3$, $x \ge 2t+1$ and set $n:= 1+2x$. Let $(\lambda_1,\dots ,\lambda _x)$ be the partition of $n$ with $\lambda_1=\lambda_2=3$ and $\lambda _i= 2$ if $3\le i\le x$. By \cite{cat}, Theorem 4.1, there is a rational normal scroll $X\subset \P^n$ such that $\dim \sigma_r(X) =-1+\sum _{i=1}^{r} \lambda_i$ for all $r\le x$. Since $\lambda_1=3$, the definition of $a_1(X)$ in  \cite{cat} p. 360, gives $\dim X=2$. Since $\lambda_2=3$, $X$ is not a cone. Hence $X$ is a smooth surface. 
		Since $2(x-t) \ge 2+2t$, $\sigma_t(X)$ has codimension at least $t+1$. Let $Z\subset X_{\reg}$ be a general union of $t$ double and let $W\subset X_{\reg}$ be a general union of $t$ 2-squares.
		By Terracini's Lemma, $h^0(\bI_Z(1)) \le n-t-1$ and $h^1(\bI_Z(1)) >0$. Hence, by Remark \ref{rangè}, we get $h^0(\bI_W(1)) >0$ and $h^1(\bI_W(1)) >0$.
		Taking $\lambda_{x+1} =1$ we get an example for $n$ even. 
		\end{exa}
		
	\section{The case of $\P^2$}\label{sec:p2}
	The main result of this section is Corollary \ref{p2<4}, stating that a general union of zero-dimensional schemes of length less or equal than 4 has good postulation in $\P^2$, except for the well-known cases of two and five double points. In order to get the result we show before that a general union of tiles in $\P^2$ has always good postulation and then that a general union of tiles and double points in $\P^2$ has good postulation except for the already mentioned cases. Finally we will conclude by using Remark \ref{cm} and Lemma \ref{2p3p}.  We will often invoke the following numerical lemma.
	\begin{lem}\label{numlem}
    For any integer $d\geq 4$ there exist $a\in\{0,1,2\}$ and $b\in\N$ such that $d+1=2a+3b$. Moreover, $a,b,d$ satisfies 
    $$2a+b\leq d,\quad\text{and}\quad \left\lfloor\frac{(d+2)(d+1)}{8}\right\rfloor\geq a+b.$$
	\end{lem}
	\proof\, The existence of $a,b$ follows by the coprimality of $2$ and $3$. Since $d+1=2a+3b$, the first inequality is equivalent to $b>0$ and this is true because $d\geq 4$ and $a\leq 2$. For the second inequality, suppose by contradiction that it is false. Then we get
	$$(d+2)(d+1)<8a+8b=4(d+1)-4b$$
	and this is a contradiction because $d\geq 4$ and $b>0$.  \prfend
	\begin{thm}\label{tiles}
	Let $d,s\in\N$ and $X\subset\P^2$ a generic union of $s$ tiles. Then $$h^0(\bI_X(d))=\max \left\{0,{d+2\choose 2}-4s\right\}.$$
	\end{thm}
	\proof\, For any $d\in\N$ we set
	$$s_*(d):=\left\lfloor\frac{{d+2\choose 2}}{4}\right\rfloor=\left\lfloor\frac{(d+2)(d+1)}{8}\right\rfloor, \quad s^*(d):=\left\lceil\frac{{d+2\choose 2}}{4}\right\rceil=\left\lceil\frac{(d+2)(d+1)}{8}\right\rceil.$$
	By Remark \ref{rangè} to prove the statement for the integer $d$ and all the integers $s$ it suffices to prove it for $s=s_*(d)$ and for $s=s^*(d)$. Moreover, by the semicontinuity theorem for cohomology it is enough to prove the statement for a specialisation of $X$. The cases $d=1$ and $d=2$ are trivial. For the rest of the proof we fix a line $L$.
	\begin{itemize}[leftmargin=*]
	\item The case $d=3$.\\
	We have $s_*(3)=2$ and $s^*(3)=3$ and we start by proving the statement for $s=2$. Let $R\neq L$ be a line of $\P^2$, $P_1\in L\setminus(L\cap R)$, $P_2\in R\setminus(L\cap R)$ and $Y_1$ and $Y_2$ two tiles such that $Y_1$ is supported at $P_1$, $Y_2$ is supported at $P_2$, $L$ is the long side of $Y_1$ and $R$ is the long side of $Y_2$. Note that 
	$$h^0(\bI_{(Y_1\cup Y_2)\cap (L\cup R),L\cup R}(3))=1$$
	and, since $\Res_{L\cup R}(Y_1\cup Y_2)=\{P_1,P_2\}$ we get
	$$h^0(\bI_{\Res_{L\cup R}(Y_1\cup Y_2)}(1))=1,\quad h^1(\bI_{\Res_{L\cup R}(Y_1\cup Y_2)}(1))=0. $$ Hence, the residual exact sequence of $Y_1\cup Y_2$ with respect to $L\cup R$ gives $h^0(\bI_{Y_1\cup Y_2}(3))=2$ and this prove the statement for $d=3$ and $s=2$. To prove the statement for $d=3$ and $s=3$ it suffices to note that, by Remark \ref{cm}, for a general curvilinear scheme $J$ with $\ell(J)=2$ we have $h^0(\bI_{Y_1\cup Y_2\cup J}(3))=0$ and there exist a tile $Y_3$ containing $J$. Hence $h^0(\bI_{Y_1\cup Y_2\cup Y_3}(3))=0$.
	\item The case $d\geq 4$.\\
	Suppose by induction that the theorem is true for any $d'<d$. By Lemma \ref{numlem} there exist $a\in\{0,1,2\}$ and $b\in\N$ such that $d+1=2a+3b$. Fix $a+b$ distinct points on $L$, say $P_1,\dots,P_a$ and $Q_1,\dots,Q_b$ and set
	$$S=\{P_1,\dots,P_a\},\quad S'=\{Q_1,\dots,Q_b\}.$$
	Since we can suppose $s\in\{s_*(d),s^*(d)\}$, by Lemma \ref{numlem} we have $s\geq a+b$. Now we consider the schemes
	$$X:=Y_1\cup\dots\cup Y_a\cup Z_1\cup\dots\cup Z_b\cup W$$
	$$Y:=Y_1\cup\dots\cup Y_a,\quad Z:=Z_1\cup\dots\cup Z_b,$$
	where:
	\begin{itemize}[leftmargin=*]
	\item $Y_i$ is a tile supported at $P_i$ having  a line $R_i\neq L$ as its long side, for $i=1,\dots,a$. Thus, $\ell(Y_i\cap R_i)=3$, $ \ell(Y_i\cap L)=2$ and  $\Res_LY_i=J_i$, where $J_i$ is a 2-jet supported at $P_i$ and lying on $R_i$;
	\item $Z_i$ is a tile supported at $Q_i$ having $L$ as long side, for $i=1,\dots,b$. Thus, $\ell(Z_i\cap L)=3$ and $\Res_L Z_i=Q_i$;
	\item $W$ is a general union of $s-a-b$ tiles away from $L$.
	\end{itemize}
	To prove the theorem it is enough to prove it for this special choice of $X$. Since $\ell(X\cap L)=d+1$, by the residual exact sequence of $X$ with respect to $L$ it is enough to prove that 
	$$h^0(\bI_{\Res_LX}(d-1))=\max\left\{0,{d+1\choose 2}-\ell(\Res_LX)\right\}.$$
	Note that 
	$$\Res_LX=J_1\cup\dots\cup J_a\cup Q_1\cup\dots\cup Q_b\cup W.$$
	 Now set $\tilde{J}:=\tilde{J_1}\cup\dots\cup\tilde{J_a}$, where $\tilde{J_i}$ is the 2-jet supported at $P_i$ and contained in $L$. Since $J_i$ is a 2-jet contained in $R_i$ and, for every $i=1,\dots,a$, $L$ is a flat limit of a line containing $P_i$, we have that $\tilde{J}$ is a flat limit of a family containing $J$. Hence, if we set $X^\prime=\tilde{J}\cup S'\cup W$, it is enough to prove that
	$$h^0(\bI_{X^\prime}(d-1))=\max\left\{0,{d+1\choose 2}-\ell(X^\prime)\right\}$$
	which is in turn equivalent to show that
	$h^0(\bI_{X^\prime}(d-1))=0$ or $h^1(\bI_{X^\prime}(d-1))=0$. Since $W$ is a general union of tiles, the inductive assumption gives that either $h^0(\bI_W(d-1))=0$ or $h^1(\bI_W(d-1))=0$ and either $h^0(\bI_W(d-2))=0$ or $h^1(\bI_W(d-2))=0$. If $h^0(\bI_W(d-1))=0$ then $h^0(\bI_{X^\prime}(d-1))=0$ and we are done, thus, from now on we suppose $h^1(\bI_W(d-1))=0$ and we distinguish two subcases:
	\begin{itemize}[leftmargin=*]
	\item Subcase $h^0(\bI_W(d-2))=0$.\\
	This assumption implies that the restriction map
	$$\rho:H^0(\bI_W(d-1))\to H^0(L,\bO_L(d-1))$$
	is injective. Since $h^1(\bI_W(d-1))=0$, to conclude this case it is enough to prove that
	$$h^0(\bI_{X^\prime}(d-1))=\max\{0,h^0(\bI_W(d-1))-\ell(\tilde{J}\cup S')\}=\max\{0,h^0(\bI_W(d-1))-2a-b\}.$$
	Consider the vector space $\im(\rho)\subseteq H^0(L,\bO_L(d-1))$  and note it depends only on $W$ and not on $X^\prime\setminus W$, so that we may take $S$ and $S^\prime$ generic after fixing $W$. Remark \ref{cm} gives
	$$\dim\im(\rho)(-\tilde{J}-S')=\max\{0,\dim\im(\rho)-2a-b\}.$$
	Hence, $h^0(\bI_{X^\prime}(d-1))=\max\{0,h^0(\bI_W(d-1))-2a-b\}$.
	\item Subcase $h^1(\bI_W(d-2))=0$.\\
	Since $L\cap W=\emptyset$, we have $\ell(X^\prime\cap L)=2a+b$ and Lemma \ref{numlem} gives $h^1(\bI_{X^\prime\cap L,L}(d-1))=0$.  Hence, the residual exact sequence of $X^\prime$ with respect to $L$ gives $h^1(\bI_{X^\prime}(d-1))=0$.
	\end{itemize}
	\end{itemize}
	The theorem is now proved. \prfend
	\begin{lem}\label{d<6}
		Let $(d,r,s)\in\N^3$ with $d\leq 6$ and let $X\subset\P^2$ be a general union of $r$ double points and $s$ tiles. Then
		$$h^0(\bI_X(d))=\max\left\{0,{d+2\choose 2}-3r-4s\right\}$$
		except for $(d,r,s)=(2,2,0),(4,5,0)$. 
	\end{lem}
	\proof\, If $s=0$ the theorem is a classical case of the Alexander-Hirschowitz theorem and if $r=0$ the theorem is Theorem \ref{tiles} so, if needed, we can assume $r\geq1$ and $s\geq 1$. For the rest of the proof we fix a line $L$. The case $d=1$ is trivial, so we distinguish the following 5 cases:
	\begin{itemize}[leftmargin=*]
		\item Case $d=2$.\\
		Since we can suppose $r\geq1$ and $s\geq 1$, it is enough to prove the statement for $(r,s)=(1,1)$. Let $Y$ be a tile supported at $P\in L$ having $L$ as its long side and $Q\in\P^2\setminus L$. The residual exact sequence of $Y\cup2Q$ with respect to $L$ gives $h^0(\bI_{Y\cup2P}(2))=h^0(\bI_{P\cup2Q}(1))=0$.
		\item Case $d=3$.\\
		Since we can suppose $r\geq 1$ and $s\geq 1$ and any tile contains a double point, by Remark \ref{rangè} it is enough to prove the statement for $(r,s)=(2,1)$. Let $Y$ be a general tile away from $L$ and $Q_1,Q_2\in L$. The residual exact sequence of $Y\cup2Q_1\cup2Q_2$ with respect to $L$ gives
		$$h^0(\bI_{Y\cup2Q_1\cup2Q_2}(3))=h^0(\bI_{Y\cup Q_1\cup Q_2}(2))=2-1-1=0$$
		where the last equality follows by $h^0(\bI_Y(2))=2$ and the fact that $Q_1$ and $Q_2$ are general.
		\item Case $d=4$.\\
		By the same argument used in the previous cases, it suffices to prove the statement for \linebreak $(r,s)=(1,3)$. Let $Y_1$ be a tile supported at $P\in L$ having $L$ as its long side, $Y_2,Y_3$ two tiles away from $L$ and $Q\in L\setminus P$. The residual exact sequence of $Y_1\cup Y_2\cup Y_3\cup2Q$ with respect to $L$ gives
		$$h^0(\bI_{Y_1\cup Y_2\cup Y_3\cup2Q}(4))=h^0(\bI_{P\cup Y_2\cup Y_3\cup Q}(3))=2-2=0$$
		where the last equality follows by Theorem \ref{tiles} and the fact that $P$ and $Q$ are general.
		\item Case $d=5$.\\
		Note that if $X$ is a general union of 5 tiles, then Theorem \ref{tiles} gives $h^0(\bI_X(5))=1$. Using the same argument of the previous cases and noting that any general tile and any general double point impose at least one additional condition, we conclude this case.
		\item Case $d=6$.\\
		By Remark \ref{rangè} it suffices to prove the statemen for a general union of 7 tiles, but this true by Theorem \ref{tiles}. \prfend
	\end{itemize}
	\begin{thm}\label{fattiles}
	Fix $(d,r,s)\in\N^3$ with and let $X\subset\P^2$ be a general union of $r$ double points and $s$ tiles. Then
	$$h^0(\bI_X(d))=\max\left\{0,{d+2\choose 2}-3r-4s\right\}$$
	except for $(d,r,s)=(2,2,0),(4,5,0)$.
	\end{thm}
	\proof\, By the same argument of Lemma \ref{d<6} we can assume $r\geq 1$ and $s\geq 1$. For any $d,s\in\N$ set
	$$r_*(d,s):=\left\lfloor\frac{{d+2\choose 2}-4s}{3}\right\rfloor,\quad r^*(d,s):=\left\lceil\frac{{d+2\choose 2}-4s}{3}\right\rceil.$$
	By Remark \ref{rangè}, given a triple $(d,r,s)$ we can suppose $r\in\{r_*(d,s),r^*(d,s)\}$ and, by Lemma \ref{d<6} we can suppose $d\geq 7$. Moreover, by the semicontinuity theorem for cohomology it is enough to prove the statement for a specialisation of $X$. Now fix $d$ and suppose by induction that the statement is true for any $d'<d$. Let $a,b\in\N$ as in Lemma \ref{numlem} such that $d+1=2a+3b$ and fix a line $L\subset\P^2$. We destinguish now 2 cases:
	\begin{itemize}[leftmargin=*]
	\item Case $s\geq a+b$.\\
	In this case it is enough to prove the statement for $X=X_1\cup X_2$ where $X_1$ is a set of $a+b$ tiles specialised on $L$ as in the proof of Theorem \ref{tiles} and $X_2$ is a general union of $r$ double points and $s-a-b$ tiles. The inductive assumption gives that $h^0(\bI_{X_2}(d-1))=0$ or $h^1(\bI_{X_2}(d-1))=0$ and $h^0(\bI_{X_2}(d-2))=0$ or $h^1(\bI_{X_2}(d-2))=0$: indeed, note that the only exceptional cases appear when $d=2$ or $d=4$, but we are assuming $d\geq7$. Hence, we can proceed as in the proof of Theorem \ref{tiles} and this concludes this case.
	\item Case $s<a+b$.\\
	Here we distinguish three subcases:
	\begin{itemize}[leftmargin=*]
		\item Subcase $d$ odd and $2r_*(d,s)\geq d+1$.\\
		Let $Q_1,\dots,Q_{(d+1)/2}$ be distinct points on $L$ and take 
		$$X=2Q_1\cup\dots\cup2Q_{(d+1)/2}\cup W.$$
		where $W$ is a general union of $r-\frac{d+1}{2}$ double points and $s$ tiles. The residual exact sequence of $X$ with respect to $L$ gives
		$$h^0(\bI_X(d))=h^0(\bI_{\Res_LX}(d-1))$$
		hence, to prove this case it suffices to prove that $h^0(\bI_{\Res_LX}(d-1))=0$ or $h^1(\bI_{\Res_LX}(d-1))=0$. We have 
		$$\Res_L X=Q_1\cup\dots\cup Q_{(d+1)/2}\cup W$$
		and by inductive assumption $h^0(\bI_{W}(d-1))=0$ or $h^1(\bI_{W}(d-1))=0$ and $h^0(\bI_{W}(d-2))=0$ or
		$h^1(\bI_{W}(d-2))=0$. Hence, since $W\cap L=\emptyset$, we can proceed as in the proof of Theorem \ref{tiles} and this concludes this case.
		\item Subcase $d$ even and $2r_*(d,s)\geq d$.\\
		Note that for the argument in this subscase it would suffices $2r_*(d,s)\geq d-2$, but we take $2r_*(d,s)\geq d$ in order to get  three mutually excluding subcases. Let $P,Q_1,\dots,Q_{d/2+1}$ be distinct points on $L$ and take
		$$X=Y\cup2Q_1\cup\dots\cup2Q_{d/2-1}\cup W$$
		where $Y$ is a tile supported at $P$ and having $L$ as its long side and $W$ is a general union of $s-\frac{d}{2}+1$ double points and $s-1$ tiles (recall that we are assuming $s\geq 1$). At this point the proof is analogous to the previous case.
		\item Subcase $r_*(d,s)\leq \frac{d}{2}-\frac{1}{2}$.\\
		Note that this is the only remaining case. Recall that, by Lemma \ref{numlem}, $2a+b\leq d$ and, since $2a+3b=d+1$ and $a\in\{0,1,2\}$, we have $3b\geq d+1$. Moreover, since we are assuming $r\in\{r_*(d,s),r^*(d,s)\}$, we also have $3r+4s\geq{d+2\choose 2}-2$. Putting all these inequalities together, and recalling that we are in the case $s<a+b$, we get:
		$${d+2\choose 2}-2\leq 3r+4s\leq \frac{3}{2}d-\frac{11}{2}+4a+4b\leq \frac{7}{2}d-\frac{11}{2}+2b\leq\frac{25}{6}d-\frac{29}{6}$$
		and this a contradiction. 	\prfend
	\end{itemize}
	\end{itemize}
	We can now prove the following corollary.
	\begin{cor}\label{p2<4}
	Let $d\in\N$ and $X\subset\P^2$ a general union of 2-squares, double points, tiles and curvilinear schemes of length less or equal than 4. Then
	$$h^0(\bI_X(d))=\max\left\{0,{d+2\choose 2}-\ell(X)\right\}$$
	except for the case where $X$ the union of two double points and $d=2$ or $X$ is the union of five double points and $d=4$.
    \end{cor}
	\proof\, It follows immediately by Remark \ref{cm}, Lemma \ref{2p3p} and Theorem \ref{fattiles}. \prfend
	 We want know to construct some examples of unions of tiles and of unions of 2-squares in $\P^2$ which have bad postulation. In particular, we construct some examples with a prescribed index of regularity and, in the case of tiles, we show that they are the only possible ones. For any positive integer $t\geq 1$ we denote by $\bA(t)$ the set of all disjoint unions of $t$ 2-squares in $\P^2$ and by $\bB(t)$ the set of all disjoint unions of $t$ tiles in $\P^2$. We recall here the following well-know fact.
	\begin{remark}\label{easy}
		Let $X\subset \P^2$ be a union of $t\geq 2$ distinct double points. Since any conic containing a 2-point $2P$ is singular at $P$ and $t\ge 2$, then $H^0(\bI_X(2)) \ne 0$ if and only if there is a line $L\subset \P^2$ such that $X\subset 2L$.
	\end{remark}
	\begin{exa}\label{new1}
		Fix a line $L\subset\P^2$ and a set $S\subset L$ such that $S$ is a set of $t\geq2$ distinct points. For each $p\in S$ fix a line $R_p$ such that $p\in R_p$ and $R_p\ne L$. Set $A_p:= 2L\cap 2R_p$ and 
		$$A:= \bigcup_{p\in S} A_p.$$
		Remark \ref{easy} gives that any union $U$ of $t$ 2-squares such that $h^0(\bI_U(2))\ne 0$ is constructed as $A$ and, obviously, $h^0(\bI_A(2)) =1$. The residual exact sequence of $A$ with respect to $2L$ gives $h^0(\bI_A(d)) \ge \binom{d}{2}$ for all $d>2$ and $h^0(\bI_A(d)) =\binom{d}{2}$ for all $t\ge (d+1)/2$.
	\end{exa}
	\begin{exa}\label{new11}
	Fix a line $L$ and a set $S\subset L$ such that $S$ is a set of $t\geq2$ distinct points. Let $B\subset \P^2$ be any union of $t$ tiles supported at $S$ and such that $L$ is a long side for each of them. Remark \ref{easy} gives that any union $U$ of $t$ tiles such that $h^0(\bI_U(2))\ne 0$ is constructed as $B$ for some $L$ and $S$ and that $h^0(\bI_B(2)) =1$. Since $\ell(L\cap B)=3t$, \eatit{we get $h^1(\bI_{B\cap L}(3t-2))>0$ and thus $h^1(\bI_B(3t-2)) >0$.} the residual exact sequence of $B$ with respect to $L$ gives $h^1(\bI_B(3t-2)) =1$ and $h^1(\bI_B(d)) =0$ for all $d\ge 3t-1$.
	\end{exa}
	Before proving the next proposition, we recall the following well-know remark.
	\begin{remark}\label{zerozero}
	Let $Z\subset \P^2$ be a zero-dimensional scheme. Then, $h^1(\bI_Z(\deg(Z) -2))\ne 0$ if and only if $Z$ is contained in a line.
	\end{remark}
	\begin{prop}\label{new2}
	Let $B\subset \P^2$ be a union of $t\ge 2$ tiles. We have $h^1(\bI_B(3t-1)) =0$ and $h^1(\bI_B(3t-2)) \ne 0$ if and only if $B$ is as in Example \ref{new11}.
	\end{prop}
	\proof\,By Example \ref{new11} it is sufficient to prove that if $h^1(\bI_B(3t-2)) \ne 0$, then $B$ is as described in Example \ref{new11}. Assume $h^1(\bI_B(3t-2)) \ne 0$.
	Take a line $L\subset \P^2$ such that $z:= \ell(L\cap B)$ is maximal. The definition of tile gives $z\ge 3$ and if $z=3$ we can also assume that $L$ is a long side of one of the tiles of $B$. We have $\ell(\Res_L(B)) =4t-z$. Note that to conclude it is sufficient to prove that $z=3t$. We assume by contradiction that $z<3t$. Since $z<3t$, then $h^1(L,\bI_{B\cap L,L}(3t-2)) =0$ and thus the residual exact sequence of $B$ with respect to $L$ gives $$h^1(\bI_{\Res_L(B)}(3t-3)) >0.$$
	First assume $t=2$. Since $h^1(\bI_{\Res_L(B)}(3)) >0$ and $\ell(\Res_L(B))=8-z\le 5$, by Remark \ref{zerozero} and by the assumption $h^1(\bI_B(3t-1))=0$, we get that $z=3$ and $\Res_L(Z)$ is contained in a line $R$. Since $z=3$, $R$ contains one of the tiles of $B$ but, since no tile is contained in a line, this is a contradiction.
	Now assume $t>2$ and that the statement is true for the union of at most $t-1$ tiles. Since $z<3t$, there is at least one line $R\ne L$ that is a long side for one of the tiles of $B$. Take one $R\ne L$ wich is a long side of one of the tiles $B$ and set $w:= \ell(R\cap B)$. First assume $$h^1(2R,\bI_{2R\cap B,2R}(3t-2)) =0.$$
	The residual exact sequence of $2R$ gives $h^1(\bI_{\Res_{2R}(B)}(3t-4)) >0$. Since $\Res_{2R}(B)$ is contained in the union of at most $t-1$ disjoint tiles
	and $3t-4=3(t-1)-1$, the inductive assumption gives a contradiction, thus 
	$$h^1(2R,\bI_{2R\cap B,2R}(3t-2))>0.$$
	Since $w\le z<3t$ we have 
	$$h^1(R,\bI_{R\cap B,R}(3t-2))=0.$$
	Hence, the residual exact sequence of $2R\cap B$ with respect to $R$ gives 
	$$h^1(\bI_{\Res_R(2R\cap B)}(3t-3)) >0.$$
	Remark \ref{zerozero} gives
	$$\ell(\Res_R(2R\cap B)) \ge 3t-2.$$
	Note that $\Res_R(2R\cap B) \subset R$ and that $$\ell(\Res_R(2R\cap B)) =\ell(2R\cap B)-\ell(R\cap B).$$
	For any zero-dimensional scheme $W$ we have $\ell(2R\cap W)\le 2\ell(R\cap W)$, thus $w\ge 3t-2$. Hence, $\ell(B)\ge \ell(2R\cap B)\ge 6t-4$, a contradiction. \prfend
	\begin{prop}\label{new2.0}
		Fix integers $t\ge 3$ and $e$ such that 
		$$\frac{1+\sqrt{288t-369}}{6} \le d\le 3t-2.$$ Then there is a union $B\subset \P^2$ of $t$ distinct tiles such that $h^1(\bI_B(d)) \ne 0$ and $h^1(\bI_B(d+1)) =0$.
	\end{prop}
	
	\begin{proof}
	Fix a line $L\subset\P^2$. We distinguish 3 cases according to class of $d\pmod{3}$.
	\begin{itemize}[leftmargin=*]
	\item Case $d\equiv 1\pmod{3}$.\\
	Take $S\subset L$ a set of $(d+2)/3$ distinct points on $L$ and let $B_1$ be a union of $(d+2)/3$ distinct tiles supported at $S$ and having $L$ as one of their long side. Set 
	$$B:= B_1\cup B_2$$ with $B_2$ general in $\bB(t-(d+2)/3)$. Since $B_2$ is general, then $B_2\cap L=\emptyset$. Now, note that:
	$$\Res_L(B)=S\cup B_2,\quad \Res_{2L}(B)=B_2$$
	and, by numerical assumption,
	$$4\left(t-\frac{d+2}{3}\right)\le \binom{d}{2}.$$
	Using twice the residual exact sequence with respect to $L$ together with these data and Theorem \ref{tiles}, it is easy to see that $h^1(\bI_B(d))=1$ and $h^1(\bI_B(d+1))=0$.
	\item Case $d\equiv2\pmod{3}$.\\
	Take $S\subset L$ a set of $(d-2)/3$ distinct points and $S'\subset L$ a set of two distinct points with $S\cap S'=\emptyset$. Let $B_1$ be a the set of $(d-2)/3+2$ distinct tiles such that $(d-2)/3$ of them are supported on the points of $S$ and have $L$ as one of their long sides and 2 of them are supported on the points of $S'$ and do not have $L$ as one of their long sides. Let $B_2\in\bB(t-(d+4)/3)$ general and take
	$$B:=B_1\cup B_2.$$
	Since $B_1\subset 3L$ to prove that $h^1(\bI_{B_1\cup B_2}(d)) =1$ and $h^1(\bI_{B_1\cup B_2}(d+1)) =0$ it is sufficient to proceed analogously to the previous case by noting that $\ell(B_2)\leq {d-1\choose 2}$ by our numerical assumption.
	\item Case $d\equiv0\pmod{3}$.\\
	Assume $d\equiv 0\pmod{3}$. In this case take $B_1$ made of $d/3$ distinct tiles supported on distinct points of $L$ and having $L$ as one of their long side, and one tile supported on another point of $L$ but not having $L$ as one their long side. Then, set
	$$B=B_1\cup B_2$$
	with $B_2\in\bB(t-d/3-1)$ general and proceed as in the previous cases. \prfend
	\end{itemize}
	\end{proof}	
	\section{The case of $\P^1\times\P^1$}\label{sec:p1p1}
	The main result of this section is Corollary \ref{p1p1<4}, stating that a general union of zero-dimensional schemes of length less or equal than 4 has good postulation in $\P^1\times\P^1$, except for the well-known cases (see \cite{laf,lp}).  In order to get the result we proceed as in the previous section: we show before that a general union of tiles in $\P^1\times\P^1$ has always good postulation and then that a general union of tiles and double points in $\P^1\times\P^1$ has good postulation except for the already mentioned cases. Finally we will conclude by using Remark \ref{cm} and Lemma \ref{2p3p}. 
	\begin{remark}\label{gtilep1p1}
	Note that if $Z\subset \P^1\times\P^1$ is a general tile supported at $P$ and $L$ is the only element of $|\bO_Y(1,0)|$ passing through $P$, then $L$ is not a long side of $P$. Thus, $\ell(Z\cap L)=2$ and $\Res_L Z=J$ is a curvilinear scheme of length 2 supported at $P$. Morever, we have $\ell(J\cap L)=1$ and $(\Res_LJ)=P$.
	\end{remark}
	\begin{lem}\label{tilesprodd<3}
	Let $Y=\P^1\times\P^1$ and $X\subset Y$ a general union of $s$ tiles and $d\in\{1,2\}$. Then, for any $e\in\N$
	$$h^0(\bI_X(d,e))=\max\{0,(d+1)(e+1)-4s\}.$$
	\end{lem}
	\proof\, Fix a point $P\in Y$ and let $L$ the unique element of $|\bO_Y(1,0)|$ passing through $P$ and $R$ the unique element of $|\bO_Y(0,1)|$ passing through $P$.
	\begin{itemize}[leftmargin=*]
		\item Case $d=1$.\\
		First assume $e=1$. By Remark \ref{rangè} it is enough to prove that if $Z$ a tile, then $h^0(\bI_Z(1,1))=0$, and we can suppose $Z$ to be supported at $P$. This is immediate by using the exact sequence of $Z$ with respect to $L$ and by Remark \ref{gtilep1p1} and Remark \ref{cm}. Now let $e>1$ and suppose by induction that the statement is true for any $(d,e')$ with $1\leq e'<e$. Consider
		$$X=Z_1\cup\dots\cup Z_s$$
		a general union of $s$ tiles, with $Z_1$ supported at $P$. By general assumption, we have that $R$ is not a long side of $Z_1$. Hence, the residual exact sequence of $X$ with respect to $R$ gives
		$$h^0(\bI_X(1,e))=h^0(\bI_{\Res_R(X)}(1,e-1)).$$ 
		Since $\Res_R(X)=J\cup Z_2\cup\dots\cup Z_s$, with $J$ a curvilinear scheme supported at $P$ with $\ell(J)=2$, the statement follows by inductive assumption and Remark \ref{cm}.
		\item Case $d=2$.\\
		The case $e=1$ is analogous to the previous case, so suppose $e>1$ and by induction that the statement is true for any $(d,e')$ with $1\leq e'<e$. Recall that, by semicontinuity of cohomology, it is enough to prove the statement for a special union $X$. Let $X=Z_1\cup\dots\cup Z_s$ with $Z_2,\dots,Z_s$ general tiles and $Z_1$ a tile supported at $P$ and such that $R$ is a long side of $Z_1$. The residual exact sequence of $X$ with respect to $R$ gives
		$$h^0(\bI_X(2,e))=h^0(\bI_{\Res_R(X)}(2,e-1)).$$ 
		Since $\Res_R(X)=P\cup Z_2\cup\dots\cup Z_s$, the statement follows by the inductive assumption.\prfend
	\end{itemize}
	\begin{thm}\label{tilesprod}
	Let $Y=\P^1\times\P^1$ and $X\subset Y$ a general union of $s$ tiles. Then, for any $(d,e)\in\N^2$
	$$h^0(\bI_X(d,e))=\max\{0,(d+1)(e+1)-4s\}.$$
	\end{thm}
	\proof\, Fix $(d,e)\in\N_{>0}^2$ and suppose by induction that the statement is true for any $(d',e)\in\N_{>0}^2$ with $1\leq d'<d$;  moreover, withouth loss of generality, we can suppose that $e\geq d$ and, by Lemma \ref{tilesprodd<3}, that $d\geq 3$. Fix a line $L\in|\bO_Y(1,0)|$ and let $a,b$ as in Lemma \ref{numlem} such that $e+1=2a+3b$. We distinguish two cases:
	\begin{itemize}[leftmargin=*]
	\item Case $s\geq a+b$.\\
	Take $P_1,\dots,P_a,Q_1,\dots,Q_b\in L$ distinct points and  $$X=Z_1\cup\dots\cup Z_a\cup Z^\prime_1\cup\dots\cup Z^\prime_b\cup W$$
	where
	\begin{itemize}[leftmargin=*]
	\item $Z_i$ is supported at $P_i\in L$ and $L$ is not a long side of $Z_i$, for $i=1,\dots,a$;
	\item $Z^\prime_i$ is supported at $Q_i\in L$ and $L$ is a long side of $Z^\prime_i$ for $i=1,\dots,b$;
	\item $W$ is a general union of $s-a-b$ tiles away from $L$.
	\end{itemize}
	The residual exact sequence of $X$ with respect to $L$ gives
	$$h^0(\bI_X(d,e))=h^0(\bI_{\Res_LX}(d-1,e))$$
	$$h^1(\bI_X(d,e))=h^1(\bI_{\Res_LX}(d-1,e))$$
	and $\Res_L X=J_1\cup\dots\cup J_a\cup Q_1\cup\dots\cup Q_b\cup W$, where $\ell(J_i)=2$  and $J_i$ is a curvilinear scheme supported at $P_i$ and not contained in $L$. The inductive assumption gives that either $h^0(\bI_W(d-1,e))=0$ or $h^1(\bI_W(d-1,e))=0$ and either $h^0(\bI_W(d-2,e))=0$ or $h^1(\bI_W(d-2,e))=0$. If $h^0(\bI_W(d-1,e))=0$ then $h^0(\bI_X(d-1,e))=0$, thus we can suppose $h^1(\bI_W(d-1,e))=0$. We distinguish two subcases:
	\begin{itemize}[leftmargin=*]
	\item Subcase $h^0(\bI_W(d-2,e))=0$.\\
	In this case we have $h^0(\bI_{\Res_LX}(d-2,e))=0$. Remark \ref{comlem0} gives $h^0(\bI_{\Res_LX}(d-1,e))=0$ or $h^1(\bI_{\Res_LX}(d-1,e))=0$. Hence, the residual exact of $X$ with respect to $L$ gives $h^0(\bI_X(d,e))=0$ or $h^1(\bI_X(d,e))=0$.
	\item Subcase $h^1(\bI_W(d-2,e))=0$.\\
	We specialise $J_i$ to $\tilde{J_i}$ with $\tilde{J_i}$ a length 2 curvilinear scheme supported at $P_i$ and contained in $L$ and we call $T$ the scheme obtained by substituing $J_i$ by $\tilde{J}_i$ in $\Res_LX$. Lemma \ref{numlem} gives
	 $$\ell(T\cap L)=2a+b\leq e<e+1=h^0(L,\bO_L(d-1,e))$$ thus $h^1(\bI_{T\cap L,L}(d-1,e))=0$. Hence, the residual exact sequence of $T$ with respect to $L$ gives $h^1(\bI_T(d-1,e))=0$ and, by semicontinuity, $h^1(\bI_{\Res_LX}(d-1,e))=0$.
	\end{itemize}
	\item Case $s<a+b$.\\
	We set $b':=\min\{b,s\}$ and if $b'=s$ we set $a'=0$ while if $b'=b$ we set $a'=a+b-s<a$. We conclude in an analogous way of the case $s\geq a+b$.\prfend
	\end{itemize}
	\begin{lem}\label{fattilesp1p1<5}
	Let $Y=\P^1\times\P^1$ and let $X\subset Y$ be a general union of $r$ double points and $s$ tiles Then, for any $d\leq4$ and $e\in\N$
	$$h^0(\bI_X(d,e))=\max\{0,(d+1)(e+1)-3r-4s\}$$
	except for the cases $(d,e,r,s)=(2,2u,2u+1,0)$ and $(d,e,r,s)=(2u,2,2u+1,0)$ for any $u\in\N_{>0}$.
	\end{lem}
	\proof\, If $s=0$ the result follows by \cite{lp} Theorem 3.1 and by \cite{laf} Table I, while if $r=0$ the result follows by Theorem \ref{tilesprod}, so, if needed, we can assume $r\geq 1$ and $s\geq 1$. Fix a point $P\in Y$ and let $L$ the unique element of $|\bO_Y(1,0)|$ passing through $P$ and $R$ the unique element of $|\bO_Y(0,1)|$ passing through $P$. Without loss of generality we can assume $e\geq d$.
	\begin{itemize}[leftmargin=*]
	\item Case $d=1$.\\
	If $e=1$ it is enough to check the good postulation of one tile and this follows by Theorem \ref{tilesprod}. Now take $e>1$ and suppose by induction that the statement is true for $d=1$ and any $e'<e$. Since $r\geq 1$ we can take $X=2P\cup X^\prime$ with $X^\prime$ a general union of $r-1$ double points and $s$ tiles away from $R$. The statement follows by the residual exact sequence of $X$ with respect to $R$ and the inductive hypothesis.
	\item Case $d=2$.\\
	If $e=1$ is enough to check the statement for a general union of one double point and one tiles, and this is straightforward. Since $s\geq 1$ we can suppose that $X=Z\cup X^\prime$ with $Z$ a tile having $R$ as one of its long sides and $X^\prime$ a general union of $r$ double points and $s-1$ tiles away from $R$. The residual exact sequence of $X$ with respect to $R$ gives
	$$h^0(\bI_X(2,e))=h^0(\bI_{\Res_RX}(2,e-1))$$
	with $\Res_RX=P\cup X^\prime$. If $s>1$ or if $s=1$ and $e$ is even then, by inductive hypotheis, $X'$ has good postulation with respect to $\bO_Y(2,e-1)$ and we can conclude. So we are left we the case $s=1$ and $e$ odd. In this case, by Remark \ref{rangè} it suffices to check the cases $r=e-1$ and $r=e$. In both cases, by \cite{lp} Theorem 3.1, it is immediate to show that the statement holds.
	\item Case $d=3$.\\
	As for the previous cases, this one can be proved by induction on $e$. For the base case $e=1$ it suffices to check it for the general union of two tiles, and this immediate by Theorem \ref{tilesprod}. For the inductive argument it is enough to take $X=2P\cup Z\cup X^\prime$, with $Z$ supported on a point of $R\setminus {P}$ and not having $R$ as one of its long sides. Then we can use the same argument of previous cases.
	\item Case $d=4$.\\
	Again we can use induction on $e$. For the base case $e=1$ it is enough to consider a general union of two double points and a tile and it is easy to check the statement in this case. Finally, for the inductive argument we proceed as in the previous case but with $Z$ specialised such that $R$ is one of its long sides. \prfend
	\end{itemize}
	\begin{thm}\label{fattilesp1xp1}
	Let $Y=\P^1\times\P^1$ and let $X\subset Y$ a general union of $r$ double points and $s$ tiles. Then
	$$h^0(\bI_X(d,e))=\max\{0,(d+1)(e+1)-3r-4s\}$$
	except for the cases $(d,e,r,s)=(2,2u,u+1,0)$ and $(d,e,r,s)=(2u,2,u+1,0)$ for any $u\in\N_{>0}$.
	\end{thm}
	\proof\,  By the same argument of Lemma \ref{fattilesp1p1<5} we can assume $r\geq 1$ and $s\geq 1$. Without loss of generality we can suppose $e\geq d$. For any $d,e,s\in\N$ set
	$$r_*(d,e,s):=\left\lfloor\frac{(d+1)(e+1)-4s}{3}\right\rfloor,\quad r^*(d,e,s):=\left\lceil\frac{(d+1)(e+1)-4s}{3}\right\rceil.$$
	By Remark \ref{rangè}, given  $(d,e,r,s)$ we can suppose $r\in\{r_*(d,e,s),r^*(d,e,r,s)\}$ and, by Lemma \ref{fattilesp1p1<5} we can suppose $d\geq 5$. Moreover, by the semicontinuity theorem for cohomology it is enough to prove the statement for a specialisation of $X$. Now fix $d$ and suppose by induction that the statement is true for any $(d',e)$ with $d'<d$. Let $a,b\in\N$ as in Lemma \ref{numlem} such that $e+1=2a+3b$ and fix a line $L\in|\bO_Y(1,0)|$. We destinguish now 2 cases:
	\begin{itemize}[leftmargin=*]
		\item Case $s\geq a+b$.\\
		In this case it is enough to prove the statement for $X=X_1\cup X_2$ where $X_1$ is a set of $a+b$ tiles specialised on $L$ as in the proof of Theorem \ref{tiles} and $X_2$ is a general union of $r$ double points and $s-a-b$ tiles. The inductive assumption gives that $h^0(\bI_{X_2}(d-1,e))=0$ or $h^1(\bI_{X_2}(d-1,e))=0$ and $h^0(\bI_{X_2}(d-2,e))=0$ or $h^1(\bI_{X_2}(d-2,e))=0$: indeed, note that the only exceptional cases appear when $(d,e)=(2,2u)$ or $(d,e)=(2u,2)$ for some $u\in\N_{>0}$, but we are assuming $d\geq5$ and $e\geq d$, so the exception cases cannot appear. Hence, we can proceed as in the proof of Theorem \ref{tilesprod} and this concludes this case.
		\item Case $s<a+b$.\\
		We set $b':=\min\{b,s\}$ and if $b'=s$ we set $a'=0$ while if $b'=b$ we set $a'=a+b-s<a$. We conclude in an analogous way of the case $s\geq a+b$. 	\prfend
	\end{itemize}
	As an immediate consequence of the previous theorem, we have the following corollary.
	\begin{cor}\label{p1p1<4}
	Let $d\in\N$ and $X\subset\P^1\times\P^1$ a general union of 2-squares, double points, tiles and curvilinear schemes of length less or equal than 4. Then
	$$h^0(\bI_X(d,e))=\max\left\{0,(d+1)(e+1)-\ell(X)\right\}$$
	except for the cases already listed in Theorem \ref{fattilesp1xp1}.
	\end{cor}
		\proof\, It follows immediately by Remark \ref{cm}, Lemma \ref{2p3p} and Theorem \ref{fattilesp1xp1}. \prfend
	To conclude, we give now an example of union of tiles in $\P^1\times\P^1$ having bad postulation.
		\begin{exa}\label{new3}
		Fix integers $t\ge 2$ and integers $d\ge e\ge 2$. Take $L\in |\bO_{\P^1\times \P^1}(1,0)|$. Fix $S\subset L$ a set of $t$ distinct points and let $B\subset \P^1\times \P^1$ be a union of $t$ distinct tiles supported at the $t$ points of $S$ and having $L$ as one of their long sides. We have
		$\ell(B\cap L)=3t$. Using twice the residual exact sequence with respect to $L$ we see that  $h^1(\bI_B(d,e)) >0$ if and only if $e\le 3t-2$ and that, if $t-1\le e\le 3t-2$, then we have $h^1(\bI_B(d,e)) =3t-1-e$.
		We leave to the interested reader the extension of Proposition \ref{new2} to the case $\P^1\times \P^1$ when $d\ge e+2$. If $d=e$ one has also to consider as long sides the elements of $|\bO_{\P^1\times\P^1}(0,1)|$. 
	\end{exa}	


\begin{thebibliography}{99}
	\bibitem{ber} J. Bertin, The punctual Hilbert scheme: an introduction. Geometric methods in representation theory. I, 1--102,
	S\'{e}min. Congr., 24-I, Soc. Math. France, Paris, 2012. 
	\bibitem{bri} J. Brian\c{c}on, Description de $\text{Hilb}^n\C\{x,y\}$. Invent. Math. 41 (1977), 45-89.
	\bibitem{cgi} S. Canino, A. Gimigliano and M. Id\`{a}, On the Jacobian scheme of a plane curve. Comm. Alg. 53 (2024), no. 2, 582-592.
	\bibitem{ccgi} S. Canino, M. V. Catalisano, A. Gimigliano and M. Id\`{a}, Superfat points and associated tensors. J. Algebra 682 (2025), 605-633.
	
	\bibitem{ccgio} S. Canino, M. V. Catalisano, A. Gimigliano, M. Id\`{a} and A. Oneto, Postulation for 2-superfat points in the plane. Ann. Univ. Ferrara Sez. VII Sci. Mat. 71 (2025), no. 4, Paper No. 59.
	
	
	\bibitem{cat} M. Catalano-Jonhson, The possible dimension of the higher secant varieties. Amer. J.
	of Math. 118 (1996), 355--361.
	
	\bibitem{cc1} L. Chiantini and C. Ciliberto, Weakly defective varieties, Trans. Amer. Math. Soc. 454 (2002), no. 1, 151--178.
	
	\bibitem{cm} C. Ciliberto and R. Miranda, Interpolations on 
	curvilinear schemes. J. Algebra 203 (1998), no. 2, 677--678.
	
	\bibitem{DeP} P. De Poi, On higher secant varieties of rational normal scrolls.
	Matematiche (Catania) 51 (1996), no. 1, 3--21.
	
	\bibitem{e} A. Eastwood, Collision de biais et application à l'interpolation. Manuscripta Math. 67 (1990), no. 2, 227-249.
	
	\bibitem{ga} F. Galuppi, Collisions of fat points and applications to interpolation theory. J. Algebra 534 (2019), 100-128.
	
	\bibitem{go} F. Galuppi and A. Oneto, Secant non-defectivity via collisions of fat points. Adv. Math. 409 (2022), 108657. 
	
	\bibitem{g1} J.-M. . Granger, Singularit\'{e}s des sch\'{e}mas de Hilbert ponctuels, Algebraic geometry (La Rábida, 1981), 130--140.
	Lecture Notes in Math., 961
	Springer-Verlag, Berlin, 1982
	
	\bibitem{g} J.-M. Granger: G\'{e}om\'etrie des sch\'{e}mas de Hilbert ponctuels, M\'{e}m. Soc. Math. France (N.S.) $2^e$ s\'erie 8, (1983), 1--84.
	
	 \bibitem{hir} A. Hirschowitz,
	La m\'{e}thode d'Horace pour l'interpolation \`{a} plusieurs variables. 
	Manuscripta Math. 50 (1985), 337--388.
	
	\bibitem{iar} A. Iarrobino, The punctual Hilbert schemes, Mem. Amer. Math. Soc. 10 (1977), no. 188, viii+112 pp.
	
	\bibitem{laf} A. Laface, On linear systems of curves on rational scrolls. Geom. Dedicata 90 (2002), 127--144.
	
	\bibitem{lp} A. Laface and E. Postinghel, Secant varieties of Segre-Veronese emveddings of $(\P^1)^r$, Math. Ann. 356 (2013), 1455-1470
\end{thebibliography}
\end{document}